\documentclass{siamltex}
\usepackage[hidelinks,breaklinks]{hyperref}
\usepackage{graphicx}
\usepackage{amsmath}
\usepackage{amssymb}
\usepackage{accents}

\renewcommand{\rank}{\mathop{\mathrm{rank}}}

\newlength{\dhatheight}

\newlength{\dtildeheight}

\title{Accurate principal component analysis via\\a few iterations
of alternating least squares}
\author{Arthur Szlam, Andrew Tulloch, and Mark Tygert}

\begin{document}

\maketitle

\begin{abstract}
A few iterations of alternating least squares with a random starting point
provably suffice to produce nearly optimal spectral- and Frobenius-norm
accuracies of low-rank approximations to a matrix;
iterating to convergence is unnecessary.
Thus, software implementing alternating least squares can be retrofitted
via appropriate setting of parameters to calculate nearly optimally accurate
low-rank approximations highly efficiently, with no need for convergence.
\end{abstract}


\section{Introduction}

Low-rank approximations are popular throughout the sciences and engineering,
often in the form of principal component analysis, and converting
any low-rank approximation to a singular value decomposition
or principal component analysis is trivial and efficient, as detailed,
for example, by~\cite{halko-martinsson-tropp}, \cite{gu}, or~\cite{woodruff}.
To calculate an accurate approximation to a matrix $A$,
we consider the low-rank approximations
\begin{equation}
A \approx S_0 T_0, \quad A \approx S_1 T_1, \quad A \approx S_2 T_2,
\quad \dots,
\end{equation}
with $S_i$ being a tall and skinny matrix
and $T_i$ being a short and fat matrix,
produced via iterations starting from $S_0$
--- iterations called ``alternating least squares''
by~\cite{young-takane-leeuw} (among others):
for each $i = 0$,~$1$,~$2$, \dots,
having $S_i$ already, we obtain $T_i$ minimizing the norm
\begin{equation}
\label{right}
\|S_i T_i - A\|,
\end{equation}
then, having $T_i$ already, we obtain $S_{i+1}$ minimizing the norm
\begin{equation}
\label{left}
\|S_{i+1} T_i - A\|,
\end{equation}
where these norms denote the spectral or Frobenius norms
(the Frobenius norm of a matrix is the square root of the sum of the squares
of the absolute values of the entries of the matrix);
specifically, we use the minimizers
\begin{equation}
\label{rightup}
T_i = S_i^{(-1)} A
\end{equation}
and
\begin{equation}
\label{leftup}
S_{i+1} = A T_i^{(-1)}
\end{equation}
for $i = 0$,~$1$,~$2$, \dots, where
\begin{equation}
\label{pseudos}
S_i^{(-1)} = (S_i^* S_i)^{-1} S_i^*
\end{equation}
and
\begin{equation}
\label{pseudot}
T_i^{(-1)} = T_i^* (T_i T_i^*)^{-1}
\end{equation}
(see Section~\ref{notation} below for precise definitions,
particularly for the inverse and pseudo\-inverse).
The appendix reviews the well-known fact that these minimize
both the spectral and Frobenius norms.

Following~\cite{halko-martinsson-tropp}, we demonstrate that
the approximations attain high accuracy after just a few of these iterations.
Specifically, the remainder of the present paper has the following structure:
Section~\ref{notation} sets notational conventions used throughout the paper.
Via mathematical analysis, Section~\ref{analysis} proves the high accuracy.
Section~\ref{examples} illustrates the high accuracy via numerical examples
with a Matlab prototype available at \url{http://tygert.com/software.html}

The accompanying prototype comes complete with a comprehensive collection
of tests, but is all in Matlab.
The present paper provides a fully rigorous basis
for more general software packages implementing alternating least squares
to be retrofitted via appropriate setting of parameters
to calculate nearly optimally accurate low-rank approximations,
with no need to wait for convergence.

\section{Notation}
\label{notation}

This section sets our notational conventions.
For any full-rank square matrix $A$, we use $A^{-1}$ to denote
the inverse of $A$.
{\it For any rank-deficient square matrix $A$, we use $A^{-1}$ to denote
the pseudoinverse of $A$}; the pseudoinverse of $A$ is the matrix representing
the inverse of $A$ with its domain restricted to the row space of $A$
(plus the identically zero map restricted to the null space of $A$).
Needless to say, for a full-rank square matrix,
the pseudoinverse is the same as the inverse.

For any matrix $A$, we denote by $A^*$ the adjoint (that is,
the conjugate transpose) of $A$,
so that the spectral norm of $A$ is given by the action of $A$ on vectors via
\begin{equation}
\label{specdef}
\|A\|_2 = \sqrt{\max_{v\;:\;v^*v=1} v^* A^* A v},
\end{equation}
and the Frobenius norm $\|A\|_F$ of $A$ is the square root of the sum
of the squares of the absolute values of the entries of $A$.
The spectral and Frobenius norms of a vector viewed as a matrix
with a single row or column are the same, and are also known as
the Euclidean norm of the vector. A definition equivalent to~(\ref{specdef}) is
\begin{equation}
\|A\|_2 = \max_{v\;:\;\|v\|_2=1} \|Av\|_2,
\end{equation}
where the norms of the vectors are the Euclidean norms.

\section{Analysis of accuracy}
\label{analysis}

This section demonstrates that the procedure --- alternating least squares ---
described in the introduction produces a highly accurate approximation
$S_i T_i$ to the given matrix $A$ even for a small number $i$ of iterations,
provided that $S_0$ is one of the random matrices
used by~\cite{halko-martinsson-tropp}
(for example, the entries of $S_0$ can be independent
and identically distributed standard normal variates).
The demonstration is simply a reduction to the proof
of accuracy for similar algorithms by~\cite{halko-martinsson-tropp};
we leave the brunt of the proof (together with a discussion
of the intuitions behind the proof) to~\cite{halko-martinsson-tropp}.
We begin by proving several lemmas.

The following lemma provides an explicit expression
for $S_i$ from the iterations in~(\ref{rightup}) and~(\ref{leftup}).
\begin{lemma}
Given matrices $A$ and $S_0$ for the iterations in~(\ref{rightup})
and~(\ref{leftup}), the matrix $S_i$ coming from those iterations
can be expressed as
\begin{equation}
\label{unrolled}
S_i = (A A^*)^i S_0 B_0 B_1 B_2 \cdots B_{i-1}
\end{equation}
for $i = 1$,~$2$, $3$, \dots, where we define
\begin{equation}
\label{auxiliary}
B_i = (S_i^* A A^* S_i)^{-1} S_i^* S_i.
\end{equation}
\end{lemma}

\begin{proof}
Combining~(\ref{rightup})--(\ref{pseudot}) and~(\ref{auxiliary}) yields
\begin{equation}
\label{recur}
S_{i+1} = A A^* S_i (S_i^* A A^* S_i)^{-1} S_i^* S_i = A A^* S_i B_i
\end{equation}
for $i = 0$,~$1$,~$2$, \dots.
Iterating the recurrence in~(\ref{recur}) yields~(\ref{unrolled}).
\end{proof}

The following lemma follows straightforwardly
from using singular value decompositions.
\begin{lemma}
\label{lemming}
Suppose that $A$, $S$, and $T$ are matrices such that
\begin{equation}
\label{solution}
T = S^{(-1)} A,
\end{equation}
where
\begin{equation}
\label{pseudo}
S^{(-1)} = (S^* S)^{-1} S^*.
\end{equation}

Then, the ranks of $S^* A$, $T$, and $A T^*$ are all equal.
\end{lemma}

\begin{proof}
Combining~(\ref{solution}) and~(\ref{pseudo}) yields that
\begin{equation}
\label{simpler}
T = (S^* S)^{-1} S^* A
\end{equation}
and
\begin{equation}
\label{simpler2}
T A^* = (S^* S)^{-1} S^* A A^*.
\end{equation}
We form the full singular value decompositions
\begin{equation}
\label{ASVD}
A = U_A \Sigma_A V_A^*
\end{equation}
and
\begin{equation}
\label{SSVD}
S = U_S \Sigma_S V_S^*,
\end{equation}
where $U_A$, $V_A$, $U_S$, and $V_S$ are unitary,
and all entries of $\Sigma_A$ and $\Sigma_S$ are nonnegative
and are zero off the main diagonals.
Combining~(\ref{simpler})--(\ref{SSVD}) yields
\begin{equation}
\label{comp1}
S^* A = V_S \Sigma_S^* U_S^* U_A \Sigma_A V_A^*,
\end{equation}
\begin{equation}
\label{comp2}
T = V_S \Sigma_S^{(-1)} U_S^* U_A \Sigma_A V_A^*,
\end{equation}
and
\begin{equation}
\label{comp3}
T A^* = V_S \Sigma_S^{(-1)} U_S^* U_A \Sigma_A^{(2)} U_A^*,
\end{equation}
where $\Sigma_S^{(-1)}$ is the same as $\Sigma_S^*$,
but replacing its nonzero diagonal entries with their reciprocals,
and where $\Sigma_A^{(2)} = \Sigma_A \Sigma_A^*$ is the square diagonal matrix
with the squares of the diagonal entries of $\Sigma_A$ on its diagonal.

Combining~(\ref{comp1})--(\ref{comp3}) and the fact that
$U_A$, $V_A$, $U_S$, and $V_S$ are unitary yields
\begin{equation}
\label{simp1}
\rank(S^* A) = \rank(\Sigma_S^* W \Sigma_A),
\end{equation}
\begin{equation}
\label{simp2}
\rank(T) = \rank(\Sigma_S^{(-1)} W \Sigma_A),
\end{equation}
and
\begin{equation}
\label{simp3}
\rank(T A^*) = \rank(\Sigma_S^{(-1)} W \Sigma_A^{(2)}),
\end{equation}
where $W$ is the unitary matrix
\begin{equation}
W = U_S^* U_A.
\end{equation}
The claim stated in the lemma (that the ranks of $S^* A$, $T$, and $A T^*$ are
all equal) then follows from the combination of~(\ref{simp1})--(\ref{simp3})
and the facts that $\Sigma_S^{(-1)} W \Sigma_A$ is the same
as $\Sigma_S^* W \Sigma_A$ with its rows rescaled by nonzero multiples
(so that they have the same row space), and that (assuming $A$ is square)
$\Sigma_S^{(-1)} W \Sigma_A^{(2)}$ is the same
as $\Sigma_S^{(-1)} W \Sigma_A$ with its columns rescaled by nonzero multiples
(so that they have the same column space);
of course, the rank of a matrix is equal to the dimension of its row space
(which is the same as the dimension of its column space).
If $A$ is not square, then either $\Sigma_S^{(-1)} W \Sigma_A^{(2)}$
is the same as $\Sigma_S^{(-1)} W \Sigma_A$ augmented by columns of zeros
and with its columns rescaled by nonzero multiples
or $\Sigma_S^{(-1)} W \Sigma_A$ is the same
as $\Sigma_S^{(-1)} W \Sigma_A^{(2)}$ augmented by columns of zeros
and with its columns rescaled by nonzero multiples
(so that again they have the same column space, which is the same
as the column space of the original, unaugmented $\Sigma_S^{(-1)} W \Sigma_A$
or $\Sigma_S^{(-1)} W \Sigma_A^{(2)}$).
\end{proof}

The adjoint of the preceding lemma is the following.

\begin{corollary}
\label{coring}
Suppose that $A$, $S$, and $T$ are matrices such that
\begin{equation}
S = A T^{(-1)},
\end{equation}
where
\begin{equation}
T^{(-1)} = T^* (T T^*)^{-1}.
\end{equation}

Then, the ranks of $A T^*$, $S$, and $S^* A$ are all equal.
\end{corollary}

The following lemma follows from using both Lemma~\ref{lemming}
and Corollary~\ref{coring}.
\begin{lemma}
\label{maxrank}
Given the matrices $A$, $T_i$, and $S_{i+1}$
from~(\ref{rightup}) and~(\ref{leftup}),
the ranks of $S_0^* A$, $T_0$, $A T_0^*$, $S_1$, $S_1^* A$, $T_1$, $A T_1^*$,
$S_2$, $S_2^* A$, $T_2$, $A T_2^*$, $S_3$, $S_3^* A$, $T_3$, $A T_3^*$, \dots\
are all equal.
\end{lemma}
\begin{proof}
This lemma follows from induction on $i = 0$,~$1$,~$2$, \dots,
using Lemma~\ref{lemming} and Corollary~\ref{coring}.
\end{proof}

The following theorem follows from~(\ref{unrolled}) and Lemma~\ref{maxrank}.
\begin{theorem}
\label{colspaces}
Given matrices $A$ and $S_0$ for the iterations in~(\ref{rightup})
and~(\ref{leftup}), the column space of the matrix $S_i$
coming from those iterations is the same as the column space
of $(A A^*)^i S_0$, for $i = 1$,~$2$, $3$, \dots.
\end{theorem}

\begin{proof}
As seen from~(\ref{unrolled}), the column space of $S_i$ is a subspace
of the column space of $(A A^*)^i S_0$.
Moreover, the row space of $(A A^*)^i S_0$ is a subspace
of the row space of $A^* S_0$, so
\begin{equation}
\label{rankrel1}
\rank((A A^*)^i S_0) \le \rank(A^* S_0).
\end{equation}

As already mentioned, the column space of $S_i$ is a subspace
of the column space of $(A A^*)^i S_0$;
if the subspace were not the whole space, then
\begin{equation}
\label{rankrel2}
\rank(S_i) < \rank((A A^*)^i S_0),
\end{equation}
and then combining~(\ref{rankrel1}) and~(\ref{rankrel2}) would yield
\begin{equation}
\rank(S_i) < \rank(A^* S_0) = \rank(S_0^* A),
\end{equation}
contradicting Lemma~\ref{maxrank}\ \ \dots\
thus, the claim stated in the theorem must be true.
\end{proof}

Finally, calculating $T_i$ minimizing~(\ref{right}) constructs
the best approximation $S_i T_i$ to $A$ such that the column space
of the approximation lies in the column space of $S_i$
--- which is the same as the column space of $(A A^*)^i S_0$,
as Theorem~\ref{colspaces} proves ---
where ``best'' means minimizing the discrepancy in the spectral norm,
which is the same as minimizing the discrepancy in the Frobenius norm,
as reviewed in the appendix.
This produces a highly accurate approximation $S_i T_i$ to $A$
even for a small number $i$ of iterations,
as proven by~\cite{halko-martinsson-tropp},
provided that $S_0$ is one of the random matrices
used by~\cite{halko-martinsson-tropp}
(for example, the entries of $S_0$ can be independent
and identically distributed standard normal variates)
--- iterating until convergence is unnecessary.

\bigskip
\bigskip
\bigskip
\bigskip
\bigskip
\bigskip

\section{Numerical examples}
\label{examples}

This section presents several numerical experiments
on an implementation in Matlab of the algorithm (alternating least squares)
discussed in the introduction. Although the numerical experiments discussed
here are somewhat limited in order to keep the presentation succinct,
the codes together with software extensively testing them are available at
\url{http://tygert.com/software.html}

We consider various values for positive integers $m$ and $n$,
as specified in the captions for Tables~\ref{numex}--\ref{numex3},
and calculate rank-$k$ approximations to the $m \times n$ matrix
\begin{equation}
A = F \Sigma G,
\end{equation}
where $F$ and $G$ are $m \times m$ and $n \times n$
unitary discrete Fourier transforms, respectively,
and $\Sigma$ is an $m \times n$ matrix whose entries are all zeros
except for the diagonal entries
\begin{equation}
\Sigma_{i,i} = \delta^{\lfloor i/2 \rfloor/(k/2)}
\end{equation}
for $i = 1$,~$2$, \dots, $k$, and
\begin{equation}
\Sigma_{i,i} = \delta \cdot \frac{\min\{m,n\}-i}{\min\{m,n\}-k-1}
\end{equation}
for $i = k+1$,~$k+2$, \dots, $\min\{m,n\}$
($\lfloor i/2 \rfloor$ is the greatest integer less than or equal to $i/2$);
the tables below specify various values for $k$ and $\delta$.
Thus, the spectral norm of $A$ is 1:
\begin{equation}
\|A\|_2 = 1.
\end{equation}

\clearpage

The headings of Tables~\ref{numex}--\ref{numex3} have the following meanings:
\begin{itemize}
\item $j$ is the number of iterations conducted.
\item $k$ is the rank of the approximation constructed
--- the number of columns in $S_i$ from~(\ref{right}),
which is also the number of rows in $T_i$ from~(\ref{right}).
\item $\delta$ is the spectral-norm accuracy
of the best possible rank-$k$ approximation.
\item $\epsilon$ is the spectral-norm accuracy of the calculated rank-$k$
approximation, with the spectral-norm accuracy computed via 100 iterations
of the power method.
\item $t$ is the time in seconds required to compute the approximation
(without using any fast Fourier transforms to leverage the special structure
of the matrix $A$).
\end{itemize}

The tables illustrate the importance of using at least one (preferably two
or more) iterations, as then the accuracy ($\epsilon$) of the computed
approximation is nearly the best possible ($\delta$).
The accuracies are indeed excellent, even with just a couple iterations.
The timings scale as expected,
roughly in proportion to the number of entries in the matrices;
we used Matlab version R2015B on an Apple MacBook Pro with a 2.6~GHz
Intel Core i7 processor.

\vfill

\begin{table}[hb]
\caption{$m = 2048$, $n=4096$}
\label{numex}
\begin{center}
{\tt
\begin{tabular}{rrccc}
$j$ & $k$ & $\delta$ & $\epsilon$ & $t$ \\\hline\hline
 0 &  2 & 1e-03 & 1.4e-02 & 5.8e-01 \\
 1 &  2 & 1e-03 & 1.0e-03 & 8.1e-01 \\
 2 &  2 & 1e-03 & 1.0e-03 & 1.2e+00 \\
10 &  2 & 1e-03 & 1.0e-03 & 4.5e+00 \\
\hline
 0 & 10 & 1e-03 & 1.8e-02 & 7.5e-01 \\
 1 & 10 & 1e-03 & 1.2e-03 & 1.4e+00 \\
 2 & 10 & 1e-03 & 1.0e-03 & 2.1e+00 \\
10 & 10 & 1e-03 & 1.0e-03 & 8.0e+00 \\
\hline
 0 &  2 & 1e-11 & 1.3e-10 & 4.1e-01 \\
 1 &  2 & 1e-11 & 1.0e-11 & 7.8e-01 \\
 2 &  2 & 1e-11 & 1.0e-11 & 1.2e+00 \\
10 &  2 & 1e-11 & 1.0e-11 & 4.3e+00 \\
\hline
 0 & 10 & 1e-11 & 2.4e-10 & 7.5e-01 \\
 1 & 10 & 1e-11 & 1.0e-11 & 1.4e+00 \\
 2 & 10 & 1e-11 & 1.0e-11 & 2.1e+00 \\
10 & 10 & 1e-11 & 1.0e-11 & 8.0e+00
\end{tabular}
}
\end{center}
\end{table}

\begin{table}
\caption{$m = 4096$, $n=4096$}
\label{numex2}
\begin{center}
{\tt
\begin{tabular}{rrccc}
$j$ & $k$ & $\delta$ & $\epsilon$ & $t$ \\\hline\hline
 0 &  2 & 1e-03 & 1.5e-02 & 9.6e-01 \\
 1 &  2 & 1e-03 & 1.0e-03 & 1.9e+00 \\
 2 &  2 & 1e-03 & 1.0e-03 & 2.8e+00 \\
10 &  2 & 1e-03 & 1.0e-03 & 1.0e+01 \\
\hline
 0 & 10 & 1e-03 & 2.2e-02 & 1.8e+00 \\
 1 & 10 & 1e-03 & 1.3e-03 & 3.6e+00 \\
 2 & 10 & 1e-03 & 1.0e-03 & 5.4e+00 \\
10 & 10 & 1e-03 & 1.0e-03 & 2.1e+01 \\
\hline
 0 &  2 & 1e-11 & 2.6e-10 & 9.6e-01 \\
 1 &  2 & 1e-11 & 1.0e-11 & 1.9e+00 \\
 2 &  2 & 1e-11 & 1.0e-11 & 2.9e+00 \\
10 &  2 & 1e-11 & 1.0e-11 & 1.0e+01 \\
\hline
 0 & 10 & 1e-11 & 4.2e-10 & 1.8e+00 \\
 1 & 10 & 1e-11 & 1.0e-11 & 3.6e+00 \\
 2 & 10 & 1e-11 & 1.0e-11 & 5.5e+00 \\
10 & 10 & 1e-11 & 1.0e-11 & 2.0e+01
\end{tabular}
}
\end{center}
\end{table}

\begin{table}
\caption{$m = 4096$, $n=8192$}
\label{numex3}
\begin{center}
{\tt
\begin{tabular}{rrccc}
$j$ & $k$ & $\delta$ & $\epsilon$ & $t$ \\\hline\hline
 0 &  2 & 1e-03 & 1.2e-02 & 2.0e+00 \\
 1 &  2 & 1e-03 & 1.0e-03 & 3.7e+00 \\
 2 &  2 & 1e-03 & 1.0e-03 & 5.7e+00 \\
10 &  2 & 1e-03 & 1.0e-03 & 2.0e+01 \\
\hline
 0 & 10 & 1e-03 & 2.1e-02 & 3.8e+00 \\
 1 & 10 & 1e-03 & 1.4e-03 & 7.5e+00 \\
 2 & 10 & 1e-03 & 1.0e-03 & 1.1e+01 \\
10 & 10 & 1e-03 & 1.0e-03 & 4.1e+01 \\
\hline
 0 &  2 & 1e-11 & 1.5e-10 & 1.8e+00 \\
 1 &  2 & 1e-11 & 1.0e-11 & 3.8e+00 \\
 2 &  2 & 1e-11 & 1.0e-11 & 5.6e+00 \\
10 &  2 & 1e-11 & 1.0e-11 & 2.1e+01 \\
\hline
 0 & 10 & 1e-11 & 5.3e-10 & 3.6e+00 \\
 1 & 10 & 1e-11 & 1.0e-11 & 7.1e+00 \\
 2 & 10 & 1e-11 & 1.0e-11 & 1.1e+01 \\
10 & 10 & 1e-11 & 1.0e-11 & 4.0e+01
\end{tabular}
}
\end{center}
\end{table}

\clearpage

\appendix
\section{Common minimizers for the spectral \& Frobenius norms}

This appendix reviews the fact that, given matrices $A$ and $S$,
one matrix $T$ minimizing the norm
\begin{equation}
\label{minobj}
\| ST-A \|,
\end{equation}
with the norm being the spectral norm or the Frobenius norm, is
\begin{equation}
\label{sol}
T = S^{(-1)} A,
\end{equation}
where the so-called ``pseudoinverse'' of $S$ is
\begin{equation}
S^{(-1)} = (S^* S)^{-1} S^*,
\end{equation}
with the inverse and pseudoinverse defined in Section~\ref{notation}.

Indeed, for the spectral norm, for any $T$ --- not just that in~(\ref{sol}),
\begin{multline}
\label{unitary}
\| ST-A \|_2^2 = \left\| (ST-A)^* (ST-A) \right\|_2 \\
= \left\| \left( \left(
\begin{array}{c} SS^{(-1)} \\\hline I-SS^{(-1)} \end{array}
\right) (ST-A) \right)^*
\left( \begin{array}{c} SS^{(-1)} \\\hline I-SS^{(-1)} \end{array}
\right) (ST-A) \right\|_2 \\
= \left\| \left( \begin{array}{c} SS^{(-1)} \\\hline I-SS^{(-1)} \end{array}
\right) (ST-A) \right\|_2^2
= \left\| \left(
\begin{array}{c} ST-SS^{(-1)}A \\\hline SS^{(-1)}A-A \end{array}
\right) \right\|_2^2.
\end{multline}
The definition of the spectral norm in~(\ref{specdef}) yields
\begin{equation}
\label{square}
\| SS^{(-1)}A-A \|_2^2 \le \left\| \left(
\begin{array}{c} ST-SS^{(-1)}A \\\hline SS^{(-1)}A-A \end{array}
\right) \right\|_2^2
\le \| ST-SS^{(-1)}A \|_2^2 + \| SS^{(-1)}A-A \|_2^2.
\end{equation}
Combining~(\ref{unitary}) and~(\ref{square}) yields
\begin{equation}
\| SS^{(-1)}A-A \|_2 \le \| ST-A \|_2
\le \sqrt{\| ST-SS^{(-1)}A \|_2^2 + \| SS^{(-1)}A-A \|_2^2},
\end{equation}
so that~(\ref{minobj}) is minimal for the spectral norm when
\begin{equation}
\label{required}
ST = SS^{(-1)}A.
\end{equation}

For the Frobenius norm, for any $T$,
\begin{equation}
\label{sum}
\|ST-A\|_F^2 = \sum_{k=1}^n \|St^{[k]}-a^{[k]}\|_2^2,
\end{equation}
where $t^{[1]}$,~$t^{[2]}$, \dots, $t^{[n]}$ are the columns of $T$,
and $a^{[1]}$,~$a^{[2]}$, \dots, $a^{[n]}$ are the columns of $A$.
Using the above result for the spectral norm with $t^{[k]}$ replacing $T$
and with $a^{[k]}$ replacing $A$, the right-hand side of~(\ref{sum}) is minimal
when
\begin{equation}
St^{[k]} = SS^{(-1)}a^{[k]}
\end{equation}
for $k = 1$,~$2$, \dots, $n$, which happens to be equivalent
to~(\ref{required}) for the full $A$ for all its columns simultaneously.

Thus, for both the spectral and Frobenius norms,
(\ref{minobj}) is minimal when~(\ref{required}) holds,
and~(\ref{required}) clearly holds for $T$ defined in~(\ref{sol}).

[A similar argument uses the identity
\begin{multline}
\label{identity}
(ST-A)^* (ST-A) - \left( (I-SS^{(-1)})A \right)^* (I-SS^{(-1)})A \\
= (ST-A)^* (ST-A) - A^*(I-SS^{(-1)})A = (ST-A)^* SS^{(-1)} (ST-A), \\
= \left( S^* (ST-A) \right)^* (S^* S)^{-1} \left( S^* (ST-A) \right),
\end{multline}
the fact that the right-hand side of~(\ref{identity}) is nonnegative definite,
and the relations
\begin{equation}
\|ST-A\|_2^2 = \max_{v\;:\;v^* v = 1} v^* (ST-A)^* (ST-A) v,
\end{equation} 
\begin{equation}
\|(I-SS^{(-1)})A\|_2^2
= \max_{v\;:\;v^* v = 1} v^* \left( (I-SS^{(-1)})A \right)^* (I-SS^{(-1)})A v,
\end{equation} 
\begin{equation}
\|ST-A\|_F^2 = \sum_{k=1}^n (e^{[k]})^* (ST-A)^* (ST-A) e^{[k]},
\end{equation}
and
\begin{equation}
\|(I-SS^{(-1)})A\|_F^2 = \sum_{k=1}^n
(e^{[k]})^* \left( (I-SS^{(-1)})A \right)^* (I-SS^{(-1)})A e^{[k]},
\end{equation} 
where $e^{[1]}$,~$e^{[2]}$, \dots, $e^{[n]}$ are the unit basis vectors,
with $e^{[k]}$ being the column vector of all zeros,
except for its $k$th entry, which is 1.]

\bibliography{arxiv}
\bibliographystyle{siam}

\end{document}